\documentclass[12pt]{article}
\usepackage{amsmath,amssymb,amsthm}

\newcommand{\R}{\mathbb{R}}
\newcommand{\Z}{\mathbb{Z}}
\newcommand{\N}{\mathbb{N}}
\newcommand{\Q}{\mathbb{Q}}
\newcommand{\T}{\mathbb{T}}
\newcommand{\D}{\mathcal{D}}
\newcommand{\B}{\mathcal{B}}
\renewcommand{\L}{\mathcal{L}}

\newcommand{\const}{{\rm const}}
\renewcommand{\div}{{\rm div}}
\newcommand{\esslim}{\mathop{\rm ess\,lim}}
\newcommand{\sign}{{\rm sign}\,}

\def\Xint#1{\mathchoice
{\XXint\displaystyle\textstyle{#1}}%
{\XXint\textstyle\scriptstyle{#1}}%
{\XXint\scriptstyle\scriptscriptstyle{#1}}%
{\XXint\scriptscriptstyle\scriptscriptstyle{#1}}%
\!\int}
\def\XXint#1#2#3{{\setbox0=\hbox{$#1{#2#3}{\int}$ }
\vcenter{\hbox{$#2#3$ }}\kern-.57\wd0}}

\def\dashint{\Xint-}
\numberwithin{equation}{section}

\theoremstyle{plain}
\newtheorem{theorem}{Theorem}[section]

\theoremstyle{definition}
\newtheorem{definition}{Definition}[section]
\newtheorem{remark}{Remark}[section]

\date\empty
\title{On the long time behavior of almost periodic entropy solutions to scalar conservation laws}
\author{Evgeny Yu.~Panov\footnote{Novgorod State University, e-mail: Eugeny.Panov@novsu.ru}}
\begin{document}

\maketitle
\begin{abstract}We found the precise condition for the decay as $t\to\infty$ of Besicovitch almost periodic entropy solutions of multidimensional scalar conservation laws. Moreover, in the case of one space variable we establish asymptotic convergence of the entropy solution to a traveling wave (in the Besicovitch norm). Besides, the flux function turns out to be affine on the minimal segment containing the essential range of the limit profile while the speed of the traveling wave coincides with the slope of the flux function on this segment.
\end{abstract}

\section{Introduction}
\label{sec1}
In the half-space $\Pi=\R_+\times\R^n$, where $\R_+=(0,+\infty)$, we consider a conservation law
\begin{equation}
\label{1}
u_t+\div_x\varphi(u)=0, \quad
u=u(t,x), \ (t,x)\in\Pi.
\end{equation}
The flux vector $\varphi(u)=(\varphi_1(u),\ldots,\varphi_n(u))$ is supposed to be merely continuous: $\varphi(u)\in C(\R,\R^n)$. Recall the notion of Kruzhkov entropy solution of the Cauchy problem for equation (\ref{1}) with initial condition
\begin{equation}
\label{1I}
u(0,x)=u_0(x)\in L^\infty(\R^n).
\end{equation}

\begin{definition}[\cite{Kr}]\label{def1}
A bounded measurable function $u=u(t,x)\in L^\infty(\Pi)$ is called an entropy solution (e.s.) of (\ref{1}), (\ref{1I}) if for all $k\in\R$
\begin{equation}
\label{2}
\frac{\partial}{\partial t} |u-k|+\div_x [\sign(u-k)(\varphi(u)-\varphi(k))]\le 0
\end{equation}
in the sense of distributions on $\Pi$ (in $\D'(\Pi)$), and
\begin{equation}
\label{I}
\esslim_{t\to 0+}u(t,\cdot)=u_0 \quad \mbox{ in } L^1_{loc}(\R^n).
\end{equation}
\end{definition}

Here $\sign u=\left\{\begin{array}{rr} 1, & u>0, \\ -1, & u\le 0
\end{array}\right.$
and relation (\ref{2}) means that for each test function $h=h(t,x)\in C_0^1(\Pi)$,
$h\ge 0$,
$$
\int_\Pi \,[|u-k|h_t+\sign(u-k)(\varphi(u)-\varphi(k))\cdot\nabla_x h]dtdx\ge 0,
$$
where $\cdot$ denotes the inner product in $\R^n$.

Taking in (\ref{2}) $k=\pm R$, where $R\ge\|u\|_\infty$, we obtain that $u_t+\div_x\varphi(u)=0$ in $\D'(\Pi)$, that is an e.s. $u=u(t,x)$ is a weak solutions of this equation as well.

The existence of e.s. of (\ref{1}), (\ref{1I}) follows from the general result of \cite[Theorem~3]{PaIzv}. In the case under consideration when the flux vector is only continuous the effect of infinite speed of propagation appears, which may even leads to the nonuniqueness of e.s. if $n>1$, see examples in  \cite{KrPa1,KrPa2,PaIzv}, where exact sufficient conditions of the uniqueness were also found. Nevertheless, if an initial function $u_0$ is periodic in $\R^n$ (~at least in $n-1$ independent directions~), then the e.s. of (\ref{1}), (\ref{1I}) is unique and $x$-periodic, see \cite{PaMax1}, as well as the more general result \cite[Theorem~11]{PaIzv}.

We will study problem (\ref{1}), (\ref{1I}) in the class of Besicovitch almost periodic functions.
Let $C_R$ be the cube
$$\{ \ x=(x_1,\ldots,x_n)\in\R^n \ | \ |x|_\infty=\max_{i=1,\ldots,n}|x_i|\le R/2 \ \}, \quad R>0.$$
We define the seminorm
$$N_1(u)=\limsup_{R\to +\infty}R^{-n}\int_{C_R} |u(x)|dx, \quad u(x)\in L^1_{loc}(\R^n).
$$
Recall (~see \cite{Bes,Lev}~) that the Besicovitch space $\B^1(\R^n)$ is the closure of trigonometric polynomials, i.e. finite sums $\sum a_\lambda e^{2\pi i\lambda\cdot x}$ with ${i^2=-1}$, $\lambda\in\R^n$, in the quotient space  $B^1(\R^n)/B^1_0(\R^n)$, where
$$
B^1(\R^n)=\{ u\in L^1_{loc}(\R^n) \ | \ N_1(u)<+\infty \}, \ B^1_0(\R^n)=\{ u\in L^1_{loc}(\R^n) \ | \ N_1(u)=0 \}.
$$
The space $\B^1(\R^n)$ is equipped with the norm $\|u\|_1=N_1(u)$ (~we identify classes in the quotient space $B^1(\R^n)/B^1_0(\R^n)$ and their representatives~). The space $\B^1(\R^n)$ is a Banach space, it is isomorphic to the completeness of the space $AP(\R^n)$ of Bohr almost periodic functions with respect to the norm  $N_1$. It is known (see for instance \cite{Bes}~) that for each function $u\in \B^1(\R^n)$ there exists the mean value
$$\bar u=\dashint_{\R^n} u(x)dx\doteq\lim\limits_{R\to+\infty}R^{-n}\int_{C_R} u(x)dx$$ and, more generally, the Bohr-Fourier coefficients
$$
a_\lambda=\dashint_{\R^n} u(x)e^{-2\pi i\lambda\cdot x}dx, \quad\lambda\in\R^n.
$$
The set
$$ Sp(u)=\{ \ \lambda\in\R^n \ | \ a_\lambda\not=0 \ \} $$ is called the spectrum of an almost periodic function  $u(x)$. It is known \cite{Bes}, that the spectrum $Sp(u)$ is at most countable.

Now we assume that the initial function $u_0(x)\in\B^1(\R^n)\cap L^\infty(\R^n)$. Let $\displaystyle I=\dashint_{\R^n} u_0(x)dx$, and $M_0$ be the smallest additive subgroup of $\R^n$ containing $Sp(u_0)$.

It was shown in \cite{PaJHDE1} that an e.s. $u(t,x)$ of (\ref{1}), (\ref{1I}) is almost periodic with respect to spatial variables. Moreover, $u(t,x)\in C([0,+\infty),\B^1(\R^n))$ (after possible correction on a set of null measure) and ${Sp(u(t,\cdot))\subset M_0}$, $\displaystyle \dashint_{\R^n} u(t,x)dx=I$ for all $t\ge 0$.
The uniqueness of e.s. $u(t,x)$ in the space $C([0,+\infty),\B^1(\R^n))$ is a consequence of the following general result  \cite[Proposition~1.3]{PaJHDE1}, which holds for arbitrary bounded and measurable initial functions.

\begin{theorem}\label{th1}
Let $u(t,x),v(t,x)\in L^\infty(\Pi)$ be e.s. of (\ref{1}), (\ref{1I}) with initial functions $u_0(x), v_0(x)\in L^\infty(\R^n)$, respectively. Then for a.e. $t>0$
\begin{equation}\label{m-contr}
N_1(u(t,\cdot)-v(t,\cdot))\le N_1(u_0-v_0).
\end{equation}
\end{theorem}

For completeness we reproduce the proof.
\begin{proof}
Applying Kruzhkov doubling of variables method, we obtain the relation (~see \cite{Kr,PaIzv}~)
\begin{equation}\label{3}
|u-v|_t+\div_x[\sign(u-v)(\varphi(u)-\varphi(v))]\le 0 \ \mbox{ in } \D'(\Pi).
\end{equation}
We choose a function $g(y)\in C_0^1(\R^n)$ such that $0\le g(y)\le 1$, and $g(y)\equiv 1$ in the cube $C_1$, $g(y)\equiv 0$ in the complement of the cube $C_k$, $k>1$, and a function $h=h(t)\in C_0^1(\R_+)$, $h\ge 0$. Applying (\ref{3}) to the test function $f=R^{-n}h(t)g(x/R)$ with $R>0$, we obtain
\begin{eqnarray}\label{3i}
\int_0^\infty\left(R^{-n}\int_{\R^n} |u(t,x)-v(t,x)| g(x/R)dx\right)h'(t)dt+ \nonumber\\ R^{-n-1}\int_\Pi \sign(u-v)(\varphi(u)-\varphi(v))\cdot\nabla_y g(x/R)h(t)dtdx\ge 0.
\end{eqnarray}
Making the change $y=x/R$ in the last integral in (\ref{3i}), we derive the estimate
\begin{eqnarray}\label{3ii}
R^{-n-1}\left|\int_\Pi\sign(u-v)(\varphi(u)-\varphi(v))\cdot\nabla_y g(x/R)h(t)dtdx\right|\le\nonumber\\ R^{-1}\|\varphi(u)-\varphi(v)\|_\infty\int_\Pi|\nabla_y g|(y)h(t)dtdy\le\frac{A}{R}\int_0^{+\infty} h(t)dt,
\end{eqnarray}
where $A=\|\varphi(u)-\varphi(v)\|_\infty\int_{\R^n} |\nabla_y g|(y)dy$. Here and below we use the notation $|z|$ for the Euclidian norm of a finite-dimensional vector $z$. Let $$ I_R(t)=R^{-n}\int_{\R^n} |u(t,x)-v(t,x)| g(x/R)dx.$$ From (\ref{3i}) and (\ref{3ii}) it follows that
$$
\int_0^{+\infty}(I_R(t)-At/R)h'(t)dt=\int_0^{+\infty}I_R(t)h'(t)dt+\frac{A}{R}\int_0^{+\infty} h(t)dt\ge 0
$$
for all $h(t)\in C_0^1((0,+\infty))$, $h(t)\ge 0$. This means that the generalized derivative $\frac{d}{dt}(I_R(t)-At/R)\le 0$, which readily implies that there exists a set $F\subset (0,+\infty)$ of full Lebesgue measure (~which can be defined as the set of common Lebesgue points of functions $I_R(t)$, $R\in\Q$~) such that $\forall t_2,t_1\in F$, $t_2>t_1$, $\forall R\in\Q$
$I_R(t_2)-At_2/R\le I_R(t_1)-At_1/R$, that is $I_R(t_2)\le I_R(t_1)+A(t_2-t_1)/R$. By the evident continuity of  $I_R(t)$ with respect to $R$ the latter relation remains valid for all $R>0$. In the limit as $F\ni t_1\to 0$ we obtain, taking into account the initial conditions for e.s. $u,v$, that  $\forall t_2=t\in F$ for all $R>0$
\begin{equation}\label{4}
I_R(t)\le I_R(0)+At/R,
\end{equation}
where $\displaystyle I_R(0)=R^{-n}\int_{\R^n} |u_0(x)-v_0(x)| g(x/R)dx$.
By the properties of $g(y)$ we find the inequalities
\begin{eqnarray*}
R^{-n}\int_{C_R}|u(t,x)-v(t,x)|dx\le I_R(t)\le \\ R^{-n}\int_{C_{kR}}|u(t,x)-v(t,x)|dx=k^n (kR)^{-n}\int_{C_{kR}}|u(t,x)-v(t,x)|dx,
\end{eqnarray*}
which imply that
\begin{equation}\label{5}
N_1(u(t,\cdot)-v(t,\cdot))\le \limsup_{R\to+\infty} I_R(t)\le k^n N_1(u(t,\cdot)-v(t,\cdot)).
\end{equation}
In view of (\ref{5}) we derive from (\ref{4}) in the limit as $R\to+\infty$ that
$N_1(u(t,\cdot)-v(t,\cdot))\le k^n N_1(u_0-v_0)$ for all $t\in F$. To complete the proof it only remains to notice that $k>1$ is arbitrary.
\end{proof}

\begin{remark} \label{rem1}
As was established in \cite[Corollary~7.1]{PaJHDE}, after possible correction on a set of null measure
any e.s.  $u(t,x)\in C(\R_+,L^1_{loc}(\R^n))$. In particular, without loss of generality, we may claim that
relation (\ref{4}) holds for all $t>0$. This implies in the limit as $R\to+\infty$ that
the statement of Theorem~\ref{th1} holds for all $t>0$ as well.  The continuity property allows also to replace the essential limit in initial condition (\ref{I}) by the usual one.
\end{remark}

The main our results are contained in the following two theorems~\ref{th2},~\ref{th3}.

\begin{theorem}\label{th2}
Assume that the following non-degeneracy condition holds for the flux components in ``resonant'' directions $\xi\in M_0$:
\begin{eqnarray}\label{ND}
\forall\xi\in M_0, \xi\not=0 \mbox{ the functions } u\to \xi\cdot\varphi(u) \nonumber\\
\mbox{ are not affine in any vicinity of } I=\overline{u_0}.
\end{eqnarray}
Then an e.s. $u(t,x)\in C([0,+\infty),\B^1(\R^n))$ satisfies the decay property
\begin{equation}\label{dec}
\lim_{t\to+\infty} u(t,\cdot)=I \ \mbox{ in } \B^1(\R^n).
\end{equation}
Condition (\ref{ND}) is precise: if it fails, then there exists an initial data $u_0\in\B^1(\R^n)\cap L^\infty(\R^n)$ with the properties $Sp(u_0)\subset M_0$, $\overline{u_0}=I$, such that the corresponding e.s. $u(t,x)$ of (\ref{1}), (\ref{1I}) does not satisfy (\ref{dec}).
\end{theorem}
\begin{remark}\label{rem2}
The decay of almost periodic e.s. was firstly studied by H. Frid \cite{Frid} in the class of Stepanov  almost periodic function. This class is natural for the case of smooth flux vector $\varphi(u)$, when an e.s. $u(t,x)$ of (\ref{1}), (\ref{1I}) exhibits the property of finite speed of propagation. The decay of such solutions was established in the stronger Stepanov norm but under rather restrictive assumptions on the dependence of the length of inclusion intervals for $\varepsilon$-almost periods of $u_0$ on the parameter $\varepsilon$.
\end{remark}

Notice that in the case of a periodic function $u_0$ the group $M_0$ coincides with the dual lattice $\L'$ to the lattice $\L$ of periods of $u_0$, and in this case theorem~\ref{th2} reduces to the following result \cite{PaNHM} (~see also the earlier paper \cite{PaAIHP}~):
\begin{theorem}\label{th2a}
Under the condition
\begin{eqnarray}\label{ND1}
\forall\xi\in \L', \xi\not=0 \mbox{ the functions } u\to \xi\cdot\varphi(u) \nonumber\\
\mbox{ are not affine in any vicinity of } I=\int_{\T^n} u_0(x)dx
\end{eqnarray}
an e.s. $u(t,x)\in C([0,+\infty),L^1(\T^n))$ satisfies the decay property
\begin{equation}\label{dec1}
\lim_{t\to+\infty}\int_{\T^n} |u(t,x)-I|dx=0.
\end{equation}
Here $\T^n=\R^n/\L$ is the $n$-dimensional torus, and $dx$ is the normalized Lebesgue measure on $\T^n$.
\end{theorem}
Remark that in the case $\varphi(u)\in C^2(\R,\R^n)$ the assertion of theorem~\ref{th2a} was established in
\cite{Daferm}.
Now we consider the case of one space variable $n=1$ when (\ref{1}) has the form
\begin{equation}\label{1-1}
u_t+\varphi(u)_x=0,
\end{equation}
where $\varphi(u)\in C(\R)$. As above, we assume that $u_0\in\B^1(\R)\cap L^\infty(\R)$ and that $M_0$ is the additive subgroup of $\R$ generated by $Sp(u_0)$. For an almost periodic function
$v(x)\in \B^1(\R)$ we denote by $S(v)$ the minimal segment $[a,b]$ containing essential values of $v(x)$. This segment can be defined by the relations
\begin{eqnarray*}
b=\min \{ \ k\in\R \ | \ (v-k)^+=\max(v-k,0)=0 \mbox{ in } \B^1(\R) \ \}, \\
a=\max\{ \ k\in\R \ | \ (k-v)^+=0 \mbox{ in } \B^1(\R) \ \}.
\end{eqnarray*}
As is easy to verify, the above minimal and maximal values exist and $a\le b$.

Our second result is the following unconditional asymptotic property of convergence of an e.s. $u(t,x)$
to a traveling wave:

\begin{theorem}\label{th3}
There is a constant $c\in\R$ (speed) and a function $v(y)\in\B^1(\R)\cap L^\infty(\R)$ (profile) such that
\begin{equation}\label{trav}
\lim_{t\to+\infty} (u(t,x)-v(x-ct))=0 \ \mbox{ in } \B^1(\R).
\end{equation}
Moreover, $Sp(v)\subset M_0$, $\bar v=I=\overline{u_0}$, and $\varphi(u)-cu=\const$ on the segment $S(v)$.
\end{theorem}

We remark, in addition to theorem~\ref{th3}, that the profile $v(y)$ of the traveling wave and, if $v\not\equiv\const$, its speed $c$ are uniquely defined. Indeed, if (\ref{trav}) holds with $v=v_1,v_2$, $c=c_1,c_2$, respectively, then
$v_1(x-c_1t)-v_2(x-c_2t)\to 0$ in $\B^1(\R)$ as $t\to+\infty$, which implies the relation
\begin{equation}\label{6}
\lim_{t\to+\infty} (v_1(y)-v_2(y+(c_1-c_2)t))=0 \ \mbox{ in } \B^1(\R).
\end{equation}
By the known property of almost periodic functions (~see, for example, \cite{Bes}~), there exists a sequence $t_r\to+\infty$ such that $v_2(y+(c_1-c_2)t_r)\mathop{\to}\limits_{r\to\infty} v_2(y)$ in $\B^1(\R)$ (this is evident if $c_1=c_2$). On the other hand, in view of (\ref{6})  $v_2(y+(c_1-c_2)t_r)\mathop{\to}\limits_{r\to\infty} v_1(y)$ in $\B^1(\R)$ and hence
$v_1=v_2$ in $\B^1(\R)$. Further, if $\Delta c=c_1-c_2\not=0$, then it follows from (\ref{6}) in the limit as  $t=t_r+h/\Delta c\to+\infty$ that $v_2(y)=v_2(y+h)$ in $\B^1(\R)$ for each $h\in\R$. Therefore,
$$
v_2(y)=\dashint_{\R}v_2(y+h)dh=\dashint_{\R}v_2(h)dh=\overline{v_2}=\const.
$$
Thus, for the nonconstant profile $v=v_2$ the speed $c_1=c_2=c$ is uniquely determined. We also remark that $\|v\|_\infty\le\|u_0\|_\infty$ because by the maximum principle
$|u(t,x)|\le\|u_0\|_\infty$ a.e. in $\Pi$.

\medskip
Theorem~\ref{th3} defines the nonlinear operator $T$ on $\B^1(\R)\cap L^\infty(\R)$, which associates an initial function $u_0$ with the profile $v(y)=T(u_0)(y)$ of the limit traveling wave for the corresponding e.s. of problem (\ref{1-1}), (\ref{1I}). In theorem~\ref{th4} below we establish that $T$ does not increase the distance in
$\B^1(\R)$.

\begin{remark}\label{rem3}
In the case $n=1$ the statement of theorem~\ref{th2} follows from theorem~\ref{th3}. Indeed, under the assumptions of theorem~\ref{th2}, $v(y)=I$ in  $\B^1(\R)$. Otherwise, $a<I<b$, where $[a,b]=S(v)$ and, by theorem~\ref{th3},
$\varphi(u)=cu+\const$ in the vicinity $(a,b)$ of $I$. But the latter contradicts to assumption (\ref{ND}) of theorem~\ref{th2}.
\end{remark}

Note that in the periodic case theorems~\ref{th3},~\ref{th4} were proved in \cite{PaMZM}.

\section{Proof of theorem~\ref{th2}}\label{sec2}

We assume firstly that the initial function is a trigonometric polynomial $\displaystyle u_0(x)=\sum_{\lambda\in\Lambda}a_\lambda e^{2\pi i\lambda\cdot x}$. Here $\Lambda=Sp(u_0)\subset\R^n$ is a finite set.  The minimal additive subgroup $M_0\doteq M(u_0)$ of $\R^n$ containing $\Lambda$ is a finite generated torsion-free abelian group and therefore it is a free abelian group of
finite rank (see \cite{Lang}). Therefore, there is a basis $\lambda_j\in M_0$, $j=1,\ldots,m$, so that every element $\lambda\in M_0$ can be uniquely represented as $\displaystyle\lambda=\lambda(\bar k)=\sum_{j=1}^m k_j\lambda_j$, $\bar k=(k_1,\ldots,k_m)\in\Z^m$. In particular, the vectors $\lambda_j$, $j=1,\ldots,m$, are linearly independent over the field of rational numbers $\Q$. We introduce the finite set $J=\{ \ \bar k\in\Z^m \ | \ \lambda(\bar k)\in\Lambda \ \}$ and represent the initial function as
$$u_0(x)=\sum_{\bar k\in J} a_{\bar k}e^{2\pi i\sum_{j=1}^m k_j\lambda_j\cdot x}, \quad a_{\bar k}\doteq a_{\lambda(\bar k)}.$$
By this representation $u_0(x)=v_0(y(x))$, where
$$
v_0(y)=\sum_{\bar k\in J} a_{\bar k}e^{2\pi i\bar k\cdot y}
$$
is a periodic function on $\R^m$ with the standard lattice of periods $\Z^m$ while $y(x)$ is a linear map from  $\R^n$ to $\R^m$ defined by the equalities $\displaystyle y_j=\lambda_j\cdot x=\sum_{i=1}^n\lambda_{ji}x_i$,  $\lambda_{ji}$, $i=1,\ldots,n$, being coordinates of the vectors $\lambda_j$, $j=1,\ldots,m$. We consider the conservation law
\begin{equation}\label{1r}
v_t+\div_y \tilde \varphi(v)=0, \quad v=v(t,y), \ t>0, \ y\in\R^m,
\end{equation}
$\tilde\varphi(v)=(\tilde\varphi_1(v),\ldots,\tilde\varphi_m(v))$, where $$\tilde\varphi_j(v)=\lambda_j\cdot\varphi(u)=\sum_{i=1}^n\lambda_{ji}\varphi_i(v)\in C(\R), \quad j=1,\ldots,m.$$
As was shown in \cite{PaMax1,PaIzv}, there exists a unique e.s. $v(t,y)\in L^\infty(\R_+\times\R^m)$ of the Cauchy problem for equation (\ref{1r}) with initial function $v_0(y)$ and this e.s. is $y$-periodic, i.e.  $v(t,y+e)=v(t,y)$ a.e. in $\R_+\times\R^m$ for all $e\in\Z^m$. Besides, in view of \cite[Corollary~7.1]{PaJHDE}, we may suppose that $v(t,\cdot)\in C([0,+\infty), L^1(\T^m))$, where $\T^m=\R^m/\Z^m$ is an $m$-dimensional torus (~which may be identified with the fundamental cube $[0,1)^m$~).
Formally, for $u(t,x)=v(t,y(x))$
\begin{eqnarray*}
u_t+\div_x\varphi(u)=v_t+\sum_{i=1}^n\sum_{j=1}^m (\varphi_i(v))_{y_j}\frac{\partial y_j(x)}{\partial x_i}=\\
v_t+\sum_{i=1}^n\sum_{j=1}^m (\varphi_i(v))_{y_j}\lambda_{ji}=v_t+\sum_{j=1}^m (\tilde\varphi_j(v))_{y_j}=0.
\end{eqnarray*}
However, these reasons are correct only for classical solutions. In the general case $v(t,y)\in L^\infty(\R_+\times\R^m)$ the range of $y(x)$ may be a proper subspace of $\R^m$
(for example, this is always true if $m>n$), and the composition $v(t,y(x))$ is not even defined. The situation is saved by introduction of additional variables $z\in\R^m$. Namely, the linear change $(z,x)\to (z+y(x),x)$ is not degenerated, i.e. it is a linear automorphism of $\R^m\times\R^n$. Since $v(t,y)$ is an e.s. of equation (\ref{1r}) considered in the extended half-space $t>0$, $(y,x)\in\R^{m+n}$, then the function $u(t,z,x)=v(t,z+y(x))$ satisfies the relations
\begin{eqnarray*}
|u-k|_t+\div_x[\sign(u-k)(\varphi(u)-\varphi(k))]=\\
|v-k|_t+\sum_{i=1}^n\sum_{j=1}^m [\sign(v-k)(\varphi_i(v)-\varphi_i(k))]_{y_j}\frac{\partial y_j(x)}{\partial x_i}=\\
|v-k|_t+\sum_{j=1}^m\sum_{i=1}^n [\sign(v-k)(\varphi_i(v)-\varphi_i(k))]_{y_j}\lambda_{ji}=\\
|v-k|_t+\sum_{j=1}^m [\sign(v-k)(\tilde\varphi_j(u)-\tilde\varphi_j(k))]_{y_j}\le 0 \ \mbox{ in } \D'(\R_+\times\R^{m+n})
\end{eqnarray*}
for all $k\in\R$. Evidently, the initial condition
$$
\lim_{t\to 0+} u(t,z,x)=u_0(z,x)\doteq v_0(z+y(x)) \ \mbox{ in } L^1_{loc}(\R^{m+n})
$$
is also satisfied, therefore $u(t,z,x)$ is an e.s. of (\ref{1}), (\ref{1I}) in the extended domain $\R_+\times\R^{m+n}$.
Since equation (\ref{1}) does not contain the auxiliary variables $z\in\R^m$, then (~cf. \cite[Theorem~2.1]{PaJHDE1}~) for all $z\in E\subset\R^m$, where $E$ is a set of full measure, the function $v(t,z+y(x))$ is an e.s. of (\ref{1}), (\ref{1I}) with initial data $v_0(z+y(x))\in\B^1(\R^n)$. Therefore, $v(t,z+y(x))=u^z(t,x)$ a.e. in $\Pi$, where, in accordance with \cite[Theorem~1.6]{PaJHDE1},
$u^z(t,x)\in C([0,+\infty),\B^1(\R^n))$ is a unique almost periodic e.s. of (\ref{1}), (\ref{1I}). Therefore, we may find a countable dense set $S\subset\R_+$ and a subset $E_1\subset E$ of full measure such that $u^z(t,x)=v(t,z+y(x))$ in $\B^1(\R)$ for all $t\in S$, $z\in E_1$.

Further, as follows from independence of the vectors $\lambda_j$, $j=1,\ldots,m$, over $\Q$, the action of the additive group $\R^n$ on the torus $\T^m$ defined by the shift transformations $T_xz=z+y(x)$, $x\in\R^n$, is ergodic, see \cite{PaJHDE1} for details. By the variant of Birkhoff individual ergodic theorem \cite[Chapter VIII]{Danf} for every $w(y)\in L^1(\T^m)$ for a.e. $z\in\T^m$ there exists the mean value
\begin{equation}\label{erg}
\dashint_{\R^n}w(z+y(x))dx=\int_{\T^m} w(y)dy.
\end{equation}

In view of (\ref{erg}), there exists a set $E_2\subset E_1$ of full measure such that for $z\in E_2$ and all $t\in S$
$$
\dashint_{\R^n}|u^z(t,x)-I|dx=\dashint_{\R^n}|v(t,z+y(x))-I|dx=\int_{\T^m} |v(t,y)-I|dy.
$$
Since $u^z(t,x)\in C([0,+\infty),\B^1(\R^n))$, $v(t,\cdot)\in C([0,+\infty),L^1(\T^m))$, while the set $S$ is dense in $[0,+\infty)$, we find that property
\begin{equation}\label{erg1}
\dashint_{\R^n}|u^z(t,x)-I|dx=\int_{\T^m} |v(t,y)-I|dy
\end{equation}
remains valid for all $t\ge 0$.
Observe that $v_0(z+y(x))\to v_0(y(x))=u_0(x)$ as $z\to 0$ in $\B^1(\R^n)$ (and even in $AP(\R^n)$).
Hence, by theorem~\ref{th1} in the limit as $E_2\ni z\to 0$ \ $u^z(t,x)\to u(t,x)$ in $C([0,+\infty),\B^1(\R^n))$, where $u(t,x)$ is the e.s. of original problem (\ref{1}), (\ref{1I}).  Therefore, relation (\ref{erg1}) in the limit as  $z\to 0$ implies the equality
\begin{equation}\label{erg2}
\dashint_{\R^n}|u(t,x)-I|dx=\int_{\T^m} |v(t,y)-I|dy.
\end{equation}
Further, for every $\bar k=(k_1,\ldots,k_m)\in\Z^m$
$$
\bar k\cdot\tilde\varphi(u)=\sum_{j=1}^m\sum_{i=1}^n k_j\lambda_{ji}\varphi_i(u)=
\lambda(\bar k)\cdot\varphi(u),
$$
where $\displaystyle\lambda(\bar k)=\sum_{j=1}^m k_j\lambda_j\in M_0$. By condition (\ref{ND}) the functions  $u\to\bar k\cdot\tilde\varphi(u)$ are not affine in any vicinity of $I=\overline{u_0}=\int_{\T^m} v_0(y)dy$. We see that non-degeneracy requirement (\ref{ND1}) is satisfied, and by \cite[Theorem~1.3]{PaNHM}
$$
\lim_{t\to+\infty}\int_{\T^m}|v(t,y)-I|dy=0.
$$
Now it follows from (\ref{erg2}) that
$$
\lim_{t\to+\infty}\dashint_{\R^n}|u(t,x)-I|dx=0,
$$
i.e. (\ref{dec}) holds.

In the general case $u_0\in\B^1(\R^n)\cap L^\infty(\R^n)$ we choose a sequence $u_{0m}$, $m\in\N$, of trigonometric polynomials converging to $u_0$ in $\B^1(\R^n)$ and such that $Sp(u_{0m})\subset M_0$, $\overline{u_{0m}}=I$ (for instance, we may choose the Bochner-Fej\'er trigonometric polynomials, see \cite{Bes}~).  Let $u_m(t,x)$ be the corresponding sequence of e.s. of (\ref{1}), (\ref{1I}) with initial data $u_{0m}(x)$, $m\in\N$. By theorem~\ref{th1} and remark~\ref{rem1} this sequence converges as $m\to\infty$ to the e.s. $u(t,x)$ of the original problem in $C([0,+\infty),\B^1(\R^n))$. We has already established that under condition (\ref{ND}) e.s. $u_m(t,x)$ satisfy the decay property
$$
\lim_{t\to+\infty} u_m(t,\cdot)=I \ \mbox{ in } \B^1(\R^n).
$$
Passing to the limit as $m\to\infty$ in this relation and taking into account the uniform convergence $u_m(t,\cdot)\mathop{\to}\limits_{m\to\infty} u(t,\cdot)$ in $\B^1(\R^n)$, we obtain (\ref{dec}).

In conclusion, we demonstrate that condition (\ref{ND}) is precise. Indeed, if this condition is violated, then there is a nonzero vector $\xi\in M_0$ such that $\xi\cdot\varphi(u)=\tau u+c$ on some segment $[I-\delta,I+\delta]$, where $\tau,c,\delta\in\R$ and $\delta>0$. Obviously, the function
$$
u(t,x)=I+\delta\sin(2\pi(\xi\cdot x-\tau t))
$$
is an e.s. of (\ref{1}), (\ref{1I}) with the periodic initial function
$u_0(x)=I+\delta\sin(2\pi(\xi\cdot x))$. We see that $\overline{u_0}=I$, $Sp(u_0)\subset \{-\xi,0,\xi\}\subset M_0$ but
the e.s. $u(t,x)$ does not converge to a constant in $\B^1(\R^n)$ as $t\to+\infty$.

The proof of theorem~\ref{th2} is complete.

\section{Proof of theorem~\ref{th3}}\label{sec3}

If the flux function $\varphi(u)$ is not affine in any vicinity of $I$, then by theorem~\ref{th2} the function   $v(y)\equiv I$, and the segment $S(v)=[I,I]=\{I\}$. Otherwise, suppose that the function $\varphi(u)$ is affine in a certain maximal interval $(a,b)$, where $-\infty\le a<I<b\le+\infty$: $\varphi(u)-cu=\const$ in $(a,b)$.

Assuming that $b<+\infty$, we define $u_+=u_+(t,x)$ as the e.s. of (\ref{1-1}), (\ref{1I}) with initial function $u_0(x)+b-I>u_0$. By the comparison principle \cite{KrPa1,KrPa2,PaMax1,PaIzv} $u_+\ge u$ a.e. in $\Pi$.
We note that $\dashint_\R (u_0(x)+b-I) dx=b$ while $\varphi(u)$ is not affine in any vicinity of $b$ (otherwise, $\varphi(u)$ is affine on a larger interval $(a,b')$, $b'>b$, which contradicts the maximality of $(a,b)$~). By theorem~\ref{th2} \  $u_+(t,\cdot)\to b$ in $\B^1(\R)$ as $t\to+\infty$, and it follows from the inequality $u\le u_+$ that $(u(t,\cdot)-b)^+\to 0$ as $t\to+\infty$ in $\B^1(\R)$. Similarly, if $a>-\infty$, then $u\ge u_-$, where $u_-=u_-(t,x)$ is an e.s. of (\ref{1-1}), (\ref{1I}) with initial function $u_0(x)+a-I<u_0$. By theorem~\ref{th2} again the function $u_-(t,\cdot)\to a$ as $t\to+\infty$ in $\B^1(\R)$ because $\dashint_\R (u_0(x)+a-I) dx=a$ while the function $\varphi(u)$ is not affine in any vicinity of $a$. Therefore, $(a-u(t,\cdot))^+\mathop{\to}\limits_{t\to+\infty} 0$ in $\B^1(\R)$. The obtained limit relations can be represented in the form
\begin{equation}\label{12a}
u(t,\cdot)-s_{a,b}(u(t,\cdot))\mathop{\to}\limits_{t\to+\infty} 0 \ \mbox{ in } \B^1(\R),
\end{equation}
where $s_{a,b}(u)=\min(b,\max(a,u))$ is the cut-off function at the levels $a,b$ (it is possible that $a=-\infty$ or $b=+\infty$).

We set $w(t,x)=s_{a,b}(u(t,x))$ and choose a strictly increasing sequence $t_k>0$ such that $t_k\to+\infty$ and $N_1(u(t_k,\cdot)-w(t_k,\cdot))\le 2^{-k}$. Since $a\le w(t,x)\le b$ while $\varphi(u)=cu+\const$ on $(a,b)$, then the e.s. of (\ref{1-1}) with initial data $w(t_k,x)$ at $t=t_k$ has the form $u=w(t_k,x-c(t-t_k))$. By theorem~\ref{th1} (with the initial time $t_k$) for all $t>t_k$
\begin{eqnarray*}
\dashint_\R |w(t,x)-w(t_k,x-c(t-t_k))|dx=\dashint_\R |s_{a,b}(u(t,x))-s_{a,b}(w(t_k,x-c(t-t_k)))|dx\\ \le
\dashint_\R |u(t,x)-w(t_k,x-c(t-t_k))|dx\le \dashint_\R |u(t_k,x)-w(t_k,x)|dx\le 2^{-k}.
\end{eqnarray*}
Substituting $t=t_l$, where $l>k$, into this inequality, we obtain
$$
\dashint_\R |w(t_l,x+ct_l)-w(t_k,x+ct_k)|dx=\dashint_\R |w(t,x)-w(t_k,x-c(t_l-t_k))]dx\le 2^{-k}.
$$
Thus, $w(t_k,x+ct_k)$, $k\in\N$, is a Cauchy sequence in $\B^1(\R)$. Therefore, this sequence converges as $k\to\infty$ to some function $v(x)\in \B^1(\R)\cap L^\infty(\R)$ in $\B^1(\R)$. It is clear that the segment $S(v)\subset [a,b]$ and therefore $\varphi(u)-cu=\const$ on $S(v)$. Since
$Sp(w(t_k,x+ct_k))=Sp(w(t_k,\cdot))\subset Sp(u(t_k,\cdot))\subset M_0$, the same inclusion holds for the limit function: $Sp(v)\subset M_0$. Finally, as follows from theorem~\ref{th1}, for $t>t_k$
\begin{eqnarray*}
\dashint_\R |u(t,x)-v(x-ct)|dx\le \\ \dashint_\R |u(t_k,x)-w(t_k,x)|dx+\dashint_\R |w(t_k,x)-v(x-ct_k)|dx=\\
\dashint_\R |u(t_k,x)-w(t_k,x)|dx+\dashint_\R |w(t_k,x+ct_k)-v(x)|dx\le \\ 2^{-k}+N_1(w(t_k,\cdot+ct_k)-v)\to 0
\end{eqnarray*}
as $t\to +\infty$ (then also $k=\max \{ \ l \ | \ t>t_l \ \}\to +\infty$). We see that relation (\ref{trav}) is satisfied.
To complete the proof of theorem~\ref{th3} it only remains to notice that
$$
 \forall t>0 \quad \overline{u(t,\cdot)}=\dashint_{\R} u(t,x)dx=I, \ \overline{v}=\dashint_\R v(x-ct)dx
$$
and (\ref{trav}) implies that $\overline{v}=I$.

\medskip
In conclusion we show that the operator $u_0\to v=T(u_0)$, defined in the Introduction, does not increase the distance in $\B^1(\R)$.

\begin{theorem}\label{th4}
Let $u_{01}(x),u_{02}(x)\in \B^1(\R)\cap L^\infty(\R)$ and $v_1=T(u_{01})(x)$, $v_2=T(u_{02})(x)$. Then
\begin{equation}\label{14}
\dashint_\R |v_1(x)-v_2(x)|dx\le \dashint_\R |u_{01}(x)-u_{02}(x)|dx.
\end{equation}
\end{theorem}

\begin{proof}
Let $u_1(t,x), u_2(t,x)\in C([0,+\infty),\B^1(\R))\cap L^\infty(\Pi)$ be e.s. of (\ref{1-1}), (\ref{1I}) with initial data $u_{01}$, $u_{02}$, respectively. By theorem~\ref{th3}
$$
\delta(t)=\dashint_\R|u_1(t,x)-v_1(x-c_1t)|dx+\dashint_\R|u_2(t,x)-v_2(x-c_2t)|dx\mathop{\to}_{t\to+\infty} 0,
$$
where $c_1,c_2$ are constants. We can choose a sequence $t_k>0$ such that $t_k\to+\infty$ as $k\to\infty$, and $N_1(v_2(x+(c_1-c_2)t_k)-v_2(x))\le 1/k$. Then, with property (\ref{m-contr}) taken into account,
\begin{eqnarray*}
\dashint_\R|v_1(x)-v_2(x)|dx=\dashint_\R|v_1(x-c_1t_k)-v_2(x-c_1t_k)|dx\le \\ \dashint_\R|v_1(x-c_1t_k)-v_2(x-c_2t_k)|dx+\dashint_\R |v_2(x-c_2t_k)-v_2(x-c_1t_k)|dx=\\
\dashint_\R|v_1(x-c_1t_k)-v_2(x-c_2t_k)|dx+\dashint_\R |v_2(x+(c_1-c_2)t_k)-v_2(x)|dx\le \\
\dashint_\R |u_1(t_k,x)-u_2(t_k,x)|dx+ \delta(t_k)+1/k\le \dashint_\R |u_{01}(x)-u_{02}(x)|dx+\delta(t_k)+1/k.
\end{eqnarray*}
In the limit as $k\to\infty$ this inequality implies (\ref{14}).
\end{proof}

\begin{remark}\label{rem4}
In view of theorem~\ref{th1} the map $F$, which associates an initial data $u_0\in\B^1(\R^n)\cap L^\infty(\R^n)$ with the e.s. $u(t,x)\in C([0,+\infty),\B^1(\R^n))$ of problem (\ref{1}), (\ref{1I}), is a uniformly continuous map from $\B^1(\R^n)$ into $C([0,+\infty),\B^1(\R^n))$. Therefore, it admits the unique continuous extension on the whole space $\B^1(\R^n)$. By analogy with \cite{BCW} the corresponding function $F(u_0)=u(t,x)\in C([0,+\infty),\B^1(\R^n))$ may be called a renormalized solution of (\ref{1}), (\ref{1I}) with possibly unbounded almost periodic initial data $u_0$. By the approximation techniques all our results can be extended to the case of renormalized almost periodic solutions.
\end{remark}

\textbf{Acknowledgement.}
This work was supported by the Ministry of Education and Science of the Russian Federation (project no. 1.445.2016/FPM) and by the Russian Foundation for Basic Research (grant no. 15-01-07650-a).


\begin{thebibliography}{99}
\itemsep=-2pt
\bibitem{Bes}
Besicovitch, A.S.: Almost Periodic Functions. Cambridge University Press (1932)
\bibitem{BCW}
B\'enilan Ph., Carrillo J., Wittbold P.:
Renormalized entropy solutions of scalar conservation laws. Ann. Scuola Norm. Sup. Pisa Cl. Sci. \textbf{29}, 313--327 (2000)
\bibitem{Daferm}
Dafermos C.M.: Long time behavior of periodic solutions to scalar conservation laws in several space dimensions.
SIAM J. Math. Anal. \textbf{45}, 2064--2070 (2013)
\bibitem{Danf}
Danford N., Schwartz J.T.: Linear Operators. General Theory (Part I). Interscience Publishers, New York-–London
(1958)
\bibitem{Frid}
Frid H.: Decay of almost periodic solutions of conservation laws. Arch. Rational Mech. Anal. \textbf{161}, 43–-64 (2002)
\bibitem{Kr}
Kruzhkov S.N.: First order quasilinear equations in several independent variables. Math. USSR Sb. \textbf{10}, 217--243 (1970)
\bibitem{KrPa1}
Kruzhkov S.N., Panov E.Yu.: First-order conservative quasilinear laws with an infinite
domain of dependence on the initial data. Soviet Math. Dokl. \textbf{42}, 316--321 (1991)
\bibitem{KrPa2}
Kruzhkov S.N., Panov E.Yu.: Osgood's type conditions for uniqueness of entropy solutions
to Cauchy problem for quasilinear conservation laws of the first order. Ann. Univ. Ferrara
Sez. VII (N.S.) \textbf{40}, 31--54 (1994)
\bibitem{Lev}
Levitan B.M.: Almost Periodic Functions. Gostekhizdat, Moscow (1953)
\bibitem{Lang}
Lang S.: Algebra, Revised 3rd ed. Springer-Verlag, New York (2002)
\bibitem{PaMax1}
Panov E.Yu.: A remark on the  theory of generalized entropy sub- and supersolutions of the
Cauchy problem for  a  first-order quasilinear equation. Differ. Equ. \textbf{37}, 272--280 (2001)
\bibitem{PaIzv}
Panov E.Yu.: On generalized entropy solutions of the Cauchy problem for a first order quasilinear equation in the class of locally summable functions. Izv. Math. \textbf{66}, 1171--1218 (2002)
\bibitem{PaJHDE}
Panov E.Yu.: Existence of strong traces for generalized solutions of multidimensional scalar conservation laws.
J. Hyperbolic Differ. Equ. \textbf{2}, 885--908 (2005)
\bibitem{PaAIHP}
Panov E.Yu.: On decay of periodic entropy solutions to a scalar conservation law. Ann I. H. Poincare-AN \textbf{30}, 997--1007 (2013)
\bibitem{PaNHM}
Panov E.Yu.: On a condition of strong precompactness and the decay of periodic entropy solutions to scalar conservation laws. Netw. Heterog. Media \textbf{11}, 349--367 (2016)
\bibitem{PaMZM}
Panov E.Yu.: Long time asymptotics of periodic generalized entropy solutions of scalar conservation laws. Math. Notes \textbf{100}, 112--121 (2016)
\bibitem{PaJHDE1}
Panov E.Yu.: On the Cauchy problem for scalar conservation laws in the class of Besicovitch almost periodic functions: Global well-posedness and decay property. J. Hyperbolic Differ. Equ. \textbf{13}, 633--659 (2016)
\end{thebibliography}
\end{document}